\documentclass[11pt,leqno]{amsart}
\usepackage{amscd}
\usepackage{color}
\usepackage{amssymb}
\usepackage{amsfonts}
\usepackage{latexsym}
\usepackage{verbatim}

\theoremstyle{plain}
\newtheorem{theorem}{Theorem}[section]

\newtheorem{lemma}[theorem]{Lemma}
\newtheorem{prop}[theorem]{Proposition}
\newtheorem{cor}[theorem]{Corollary}
\newtheorem{rem}[theorem]{Remark}

\renewcommand{\b}{\begin{equation}}
\newcommand{\e}{\end{equation}}

\newcommand\R{{\mathbb R}}

\title[A Calabi-type flow in Hermitian Geometry]{A scalar Calabi-type flow in Hermitian Geometry: Short-time existence and stability}
\thanks{This work was supported by the project FIRB ``Geometria differenziale e teoria geometrica delle funzioni'',
 the project PRIN
\lq\lq  Variet\`a reali e complesse: geometria, topologia e analisi armonica" and by GNSAGA of INdAM}
\address{Dipartimento di Ingegneria e Scienze dell'Informazione e Matematica \\ Universit\`a dell'Aquila\\
via Vetoio\\ 67100 L'Aquila\\ Italy}
\email{lucio.bedulli@univaq.it}
\address{Dipartimento di Matematica G. Peano \\ Universit\`a di Torino\\
Via Carlo Alberto 10\\
10123 Torino\\ Italy}
 \email{luigi.vezzoni@unito.it}

\author{Lucio Bedulli and Luigi Vezzoni}
\date{\today}

\begin{document}
\maketitle

\begin{abstract}
We introduce a new geometric flow of Hermitian metrics which evolves an initial metric along 
the second derivative of the Chern scalar curvature. The flow depends on the choice of a background  
metric, it always reduces to a scalar equation and preserves some special classes of Hermitian structures, such as balanced and Gauduchon metrics. 
We show that the flow has always a unique short-time solution and we provide a stability result 
 when the background metric is K\"ahler with constant scalar curvature (cscK). The main theorem is obtained by proving a general result about stability of parabolic flows on Riemannian manifolds which is interesting in its own right and in particular implies the stability of the classical Calabi flow near cscK metrics.
\end{abstract}

\section{introduction}
Giving an Hermitian metric on an $2n$-dimensional complex manifold $(M,J)$ is equivalent to assigning an $(n-1,n-1)$-form $\varphi$ which is {\em positive} in the sense that 
$$
\varphi(Z_1,\dots, Z_{n-1},\bar Z_1,\dots, \bar Z_{n-1})>0
$$
for every $\{Z_1,\dots,Z_{n-1}\}$ linearly independent vector fields of type $(1,0)$ on $(M,J)$.  Indeed if such a $\varphi$
is given, there exists a unique Hermitian metric $g$ whose fundamental form $\omega$ satisfies $*_{\omega}\omega=\varphi$, where $*_\omega$ is the induced \lq\lq star\rq\rq\, Hodge operator. This point of view suggests to consider special Hermitian metrics by imposing restrictions on the derivatives of $\omega^{n-1}$, instead of $\omega$. For instance, an Hermitian metric is called {\em balanced} \cite{M} if 
$$
d\omega^{n-1}=0\,,
$$ 
{\em Gauduchon} \cite{gaud} if  
$$
\partial \bar\partial \omega^{n-1}=0
$$
and {\em strongly Gauduchon} \cite{popovici} if 
$$
\partial \omega^{n-1} \mbox{ is $\bar\partial$-exact}\,. 
$$

Given  an Hermitian form $\omega$ on a complex manifold $M$, we consider the set of functions  
$$
C^{\infty}_{\omega}(M)=\left\{v\in C^{\infty}(M)\,\,:\,\,  \omega^{n-1}+\sqrt{-1}\partial \bar \partial(v \omega^{n-2})>0\right\}\,.
$$
Every $v\in C^{\infty}_{\omega}(M)$ induces the Hermitian form $\omega_v$ defined  by 
$\omega_v^{n-1}=\omega^{n-1}+\sqrt{-1}\partial \bar \partial(v \omega^{n-2})$; we denote by $\mathcal{C}_{ \omega}(M)$ the set 
$$
\mathcal{C}_{ \omega}(M)=\left\{\omega_v\,\,:\,\,  v \in C^{\infty}_{\omega}(M)  \right\}\,.
$$
Note that $\{\omega_v^{n-1}: v \in C^\infty_{ \omega}(M)\}$ is contained in
$$
\mathcal{K}_{\omega}=\{\omega^{n-1}+\sqrt{-1}\partial \bar\partial \sigma>0\}\,. 
$$
When $\omega$ is balanced, $\mathcal{K}_{\omega}$  is the set of closed positive $(n-1,n-1)$-forms on $M$ belonging to the Bott-Chern cohomology class of $\omega^{n-1}$. 

In this paper we are interested in solutions $\omega_t \in \mathcal{C}_{ \omega}(M)$ of the geometric flow 
\begin{equation}\label{flow}
\partial_t\omega_t^{n-1}=\sqrt{-1}\partial \bar \partial(s_{\omega_t}\omega^{n-2})\,,\quad 
 \omega_{|t=0}=\omega_0
\end{equation}
whose definition depends on the background Hermitian form $\omega$. By $s_{\omega_t}$ we denote the scalar curvature of the Chern connection induced by $\omega_t$. If the background metric is K\"ahler then Hermitian metrics with constant Chern scalar curvature  are stationary solutions to the flow. 
  
Equation \eqref{flow} preserves the balanced, the Gauduchon and the strongly Gauduchon condition and 
can be reduced to the scalar equation 
$$
\partial_tu_t=s_{u_t}\,\quad u_{|t=0}=u_0\,,
$$
where $u_0\in C^{\infty}_{\omega}(M)$ is such that $\omega_{u_0}=\omega_0$  and for $v\in C_{\omega}^{\infty}(M)$, $s_v$ is the Chern scalar curvature of $\omega_{v}$.

Our main result is the following  
\begin{theorem}[Short-time existence and stability of the flow] \label{main}
Flow \eqref{flow} has always a unique short-time solution $\{\omega_t\}_{t\in [0,T_{max})}$.
Assume further that the background metric $\omega$ is  K\"ahler with constant scalar curvature. Then if $\omega_0$ is close enough to $\omega$ in $C^{\infty}$-topology, the solution $\{\omega_t\}$ is defined for any positive $t$ and converges in $C^{\infty}$-topology to $\omega$. 
\end{theorem}

The short-time existence of a solution to \eqref{flow} is obtained by proving that the operator $v\mapsto s_v$ is elliptic in a very strong sense and then applying a general result in \cite{HP,MM} (see sections \ref{section2}, \ref{section3}).

About the stability of \eqref{flow} near K\"ahler metrics with constant scalar curvature, 
we prove a general theorem about the stability of scalar flows and then we show that our flow satisfies all the assumptions of the theorem. Namely we consider the following set-up:

Let $(M, g)$ be an oriented compact Riemannian manifold with volume form $dV_g$ and let  $W^{2r,2}_+(M)$ be an open neighbourhood of $0$ in $W^{2r,2}(M)$ which is invariant by additive constants.  For $k>2r$ we denote $W^{2r,2}_+(M)\cap 
W^{k,2}(M)$ by $W^{k,2}_+(M)$ and let $C^{\infty}_+(M)=C^{\infty}(M)\cap W^{2r,2}_+(M)$.  
Let $Q\colon W^{2r,2}_+(M)\to L^2(M)$ be a smooth {\em elliptic} operator of order $2r$ (elliptic in a strong sense explained in section \ref{section2}) and denote  by $L$ the differential of $Q$ at $0$.
%[forse troppo piccolo, $k\geq 4r$]. 
Assume further that $Q$ satisfies the following conditions
\begin{enumerate}
\item[(h1)] $Q(0)=0$ and $Q(v)=Q(v+a)$ for every $v\in W_+^{2r,2}(M)$, $a\in \R$ (here the set of constant functions is identified with $\R$);

\vspace{3pt}
\item[(h2)] The kernel of $L$ is made only by constant functions and $L(W^{2r,2}_0(M)) \subseteq L^2_0(M)$; where the subscript $0$ means that the elements have average 0 with respect to $g$;

\vspace{3pt}
\item[(h3)] $L$ is  {\em symmetric} and {semi-negative definite } with respect to the $L^2$-scalar product induced by the fixed metric $g$ on $M$, i.e. 
$$
\int_M L(v_1)v_2\,dV_g=\int_M L(v_2)v_1dV_g\,\,\,\, \mbox{and } \int_M L(v_1)v_1\,dV_g\leq 0
$$
for every $v_1,v_2\in W^{2r,2}(M)$.
\end{enumerate}

Under these hypotheses we have  

\begin{theorem}[Stability]\label{stability} 
For every $\epsilon>0$ there exists $\delta>0$ such that if $u_0\in C^{\infty}_+(M)$ satisfies $\|u_0\|_{C^\infty}< \delta $, then the parabolic problem
\begin{equation}\label{parabolic}
\partial_tu_t=Q(u_t)\,,\quad 
u_{|t=0}=u_0
\end{equation}
has a unique solution $u\in C^{\infty}(M\times[0,\infty))$ such that $u_t\in C^{\infty}_+(M)$ for every $t$ and satisfies 
\begin{enumerate}
\item[1.] $\|u_t\|_{C^\infty}<\epsilon $ for every $t\in [0,\infty)$;
\vspace{3pt}
\item[2.] $u_t$ converges in $C^\infty$--topology to a smooth function $u_\infty$ such that $Q(u_\infty)=0$. 
\end{enumerate}
\end{theorem}

From theorem \ref{stability} it easily follows the stability of flow \eqref{flow} and the stability of the classical Calabi flow near cscK metrics which was already proved in \cite{Chen-He}. 

\bigskip
\noindent {\em{Acknowledgements}}. The authors would like to thank  Ernesto Buzano,  Jason Lotay, Weiyong He, Carlo Mantegazza, Valentino Tosatti  and Frederik Witt for useful conversations and remarks.

\section{Preliminaries on parabolic flows on Riemannian manifolds}\label{section2}
In this section we recall some results proved in \cite{HP,MM} about the short-time existence of parabolic flows on compact Riemannian manifolds. 

\medskip 
Let $(M,g)$ be a compact $m$-dimensional Riemannian manifold and let $Q\colon C^{\infty}(M)\to C^{\infty}(M)$ be a quasi-linear partial differential operator of order $2r$.  
Therefore $Q(v)$ locally writes as 
$$
Q(v)(x)=A^{i_1\dots i_{2r}}(x,v,\nabla v,\dots, \nabla^{2r-1}v)\nabla^{2r}_{i_1\dots i_{2r}}v(x)+b(x,v,\nabla v,\dots, \nabla^{2r-1}v)
$$
where $\nabla$ is the Levi-Civita connection of $g$ and the functions $A^{i_1\dots i_{2r}}$ and $b$ are smooth in their entries. We further assume that $Q$ is {\em elliptic} in a very strong sense  by requiring 
$$
A^{i_1j_1\dots i_rj_r}=(-1)^{r-1}E_1^{i_1j_1}\dots E^{i_rj_r}_r
$$
where each $E_k$ is a tensor of type $(2,0)$ for which there exists a positive $\lambda\in \R$ such that  
\begin{equation}\label{elliptic}
E_{k}^{ij}(x,\psi_1,\dots,\psi_{2r-1})\xi_i\,\xi_j \geq  \lambda |\xi|_{g}^2\,\,\mbox{ for every }\xi \in T^*_xM
\end{equation}
when  $x\in M$ and  $\psi_k \in \otimes_k T_x^*M$. Given such a $Q$ and an initial datum $u_0\in C^{\infty}(M)$, we consider the parabolic problem 
\begin{equation}\label{pproblem}
\partial_tu_t=Q(u_t)\,,\quad u_{|t=0}=u_0\,.
\end{equation}
We recall the following theorem whose proof can be found in \cite{HP,MM}
\begin{theorem}\label{short time existence}
Equation \eqref{pproblem} has always a maximal solution  
 $u\in C^{\infty}(M\times [0,T_{\max}))$, for some $T_{\max}>0$.  Moreover the solution $u$ depends continuously on the initial datum $u_0$. 
\end{theorem}

Let $k\in \mathbb{N}$ and $[t_1,t_2] \subset \R_{\geq 0}$. Then the {\em parabolic Sobolev spaces} $P^{k}([t_1,t_2])$ are defined as the completion of $C^{\infty}(M\times [t_1,t_2])$ with respect to the norm    
$$
\|f\|_{P^k([t_1,t_2])}^2=\sum_{l,s\in \mathbb N\,,\,\,2lr+s\leq 2rk}\int_{t_1}^{t_2}\int_M |\partial_t^l\nabla^sf|^2\,dV_g\, dt\,. 
$$ 

We recall the following theorem (see \cite[theorem 7.14]{HP} and \cite[proposition 2.3 and lemma 2.5]{MM}).

\begin{theorem}\label{IFT}
Assume 
\begin{equation}\label{condizionek}
k>\frac{m + 6r - 2}{4r}\,,
\end{equation}
and let $u\in P^{k}([0,T])$.  
Then $\left (u_{|t=0},\partial_t u-Q(u)\right)$ belongs to $W^{r(2k-1),2}(M)\times P^{k-1}([0,T])$ and the map 
$$
\mathcal{F}\colon P^k([0,T])\to W^{r(2k-1),2}(M)\times P^{k-1}([0,T])
$$
defined as  
$$
\mathcal F(u)=\left (u_{|t=0},\partial_t u-Q(u)\right).
$$
is $C^1$ and its differential $d\mathcal F_{|u}$ is an isomorphism for every $u\in P^k([0,T])$. 
\end{theorem}

\begin{rem}
{\em In the next part of the paper we need to apply theorem \ref{IFT} to operators defined on open subsets of $W^{k,2}(M)$ instead of  on the whole $W^{k,2}(M)$. The results described in this section can be easily adapted to this slightly more general setting.}
\end{rem}

In the sequel we will need the following corollary of theorem \ref{IFT}.
\begin{cor}\label{tracetheorem}
Let $k$ be an odd multiple of $r$ and $T\in \mathbb R^+$.  For every $\epsilon \in (0,T)$ there exists $C=C(k,\epsilon)$ such that 
if 
$u \in P^{\frac{k+r}{2r}}([0,T])$, then
$$
\|u_t\|_{_{W^{k,2}}}\leq C \|u\|_{P^{\frac{k+r}{2r}}([t,T])}\,,
$$
for every $t\in [0,T-\epsilon]$.
\end{cor}
\begin{proof}
By proposition 2.3 in \cite{MM} we have that if $u \in P^{\frac{k+r}{2r}}([0,\tau])$ then there exists $C(\tau)>0$ such that
$$
\|u_0\|_{_{W^{k,2}}}\leq C(\tau) \|u\|_{P^{\frac{k+r}{2r}}([0,\tau])}\,,
$$
On the other hand it is clear that $C(\tau)$ can be chosen so that $C$ is decreasing with $\tau$ and that the previous estimate is translation invariant,
hence the estimate
$$
\|u_0\|_{_{W^{k,2}}}\leq C(\epsilon) \|u\|_{P^{\frac{k+r}{2r}}([0,T-t])}\,,
$$
for every $t\in [0,T-\epsilon]$ implies the statement.
\end{proof}

\section{Short-time existence of flow \eqref{flow}}\label{section3}
In the set-up of the introduction: let $(M,g)$ be an Hermitian manifold with fundamental form $\omega=\sqrt{-1}g_{r\bar s} dz^r\wedge d\bar z^{ s}$. Then, as we have already mentioned in the introduction, the geometric flow \eqref{flow}
can be reduced to a scalar flow by using the substitution 
$$
\omega_t^{n-1}=\omega^{n-1}+\sqrt{-1}\partial \bar\partial(u_t\omega^{n-2})
$$
which leads to the evolution equation
\b\label{scalarflow}
\partial_tu_t=s_{u_t}\,,\quad u_{|t=0}=u_0 \,.
\e

We have the following 
\begin{prop}
$v\mapsto s_v$ is a $4^{th}$-order quasi-linear elliptic operator.
\end{prop}
\begin{proof}
For $v$ in $C^{\infty}_{\omega}(M)$, we have 
$$
s_v=-g_{v}^{\bar k r}\partial_r\partial_{\bar k}\,{\rm log}({\rm det}(g_{v}))
$$
where $g_v$ is the Hermitian metric induced by $v$. 
Following \cite{FWW}, we write 
$$
\omega_v^{n-1}=\left[(\sqrt{-1})^{n-1}(n-1)!\right]\, \sum_{i,j=1}^n\psi_{i\bar j}\varepsilon_{i\bar j}\,dz^1\wedge d\bar z^{ 1}\wedge \dots \wedge \widehat{dz^i}\wedge \dots\wedge  \widehat{d\bar z^j}
\wedge \dots \wedge d\bar z^n
$$
where $\varepsilon_{i\bar j}$ is $1$ if $i\leq j$ and it is $-1$ otherwise (the symbol $\,\,\widehat{ }\,\,$ means that the term is omitted). Since 
$$
({\rm det}(g_{v}))=({\rm det}(\psi))^{\frac{1}{n-1}}
$$
we have 
$$
s_v=-\frac{1}{n-1}g_{v}^{\bar k r}\partial_r\partial_{\bar k}\,{\rm log}({\rm det}(\psi))\,.
$$
Now 
$$
\partial_r\partial_{\bar k}\,{\rm log}({\rm det}(\psi))=\partial_r\left(\psi^{\bar ba }\partial_{\bar k}\,(\psi_{a\bar b})\right)
=\psi^{\bar ba }\partial_r\partial_{\bar k}\,(\psi_{a\bar b})-\psi^{\bar b l}\partial_{r}(\psi_{l\bar m})\psi^{\bar m a}\partial_{\bar k}\,(\psi_{a\bar b})
$$
and therefore
$$
s_{v}=-\frac{1}{n-1}g_{v}^{\bar k r}\psi^{\bar b a}\partial_r\partial_{\bar k}\,(\psi_{a\bar b})
+\frac{1}{n-1}g_{v}^{\bar k r}\psi^{\bar b l}\partial_{r}(\psi_{l\bar m})\psi^{\bar m a}\partial_{\bar k}\,(\psi_{a\bar b})\,.
$$
%({\rm l.o.t.} means \lq\lq lower order terms\rq\rq). 
Furthermore we have  
$$
\psi_{i\bar j}=\sum_{\sigma, \tau\in \sigma_{n-1}} |\sigma|\,|\tau|\,\left(v_{\alpha^i_{\sigma(1)}\bar{\alpha}^j_{\tau(1)}}
g_{\alpha^i_{\sigma(2)}\bar{\alpha}^j_{\tau(2)}}\dots g_{\alpha^i_{\sigma(n-1)}\bar{\alpha}^j_{\tau(n-1)}}
\right)+{\rm l.o.t}
$$
where 
$$
(\alpha^i_1,\dots ,\alpha^i_{n-1})=(1,\dots,\hat i,\dots,n)
$$
and {\rm l.o.t.} means \lq\lq lower order terms\rq\rq. 
Now we fix a point $p$ in  $M$ and holomorphic coordinates $\{z^r\}$ around $p$ such that 
$$
g_{h\bar k}=\delta_{hk}\,, \quad (g_v)_{h\bar k}=\gamma_h\delta_{hk}\,,\mbox{ at }p
$$
where $\gamma_h$ are some positive real constants. We easily get 
$$
\psi^{i\bar j}=\lambda_i\delta_{ij} 	\mbox{ at } p\,,  
$$
where $\lambda_i$ are positive and depend on $\gamma_h$.  
An easy computation yields  
$$
s_{v}=-\frac{1}{n-1}\sum_{r,k=1}^n \sum_{i\neq k}\frac{\lambda_i}{\gamma_r} v_{r\bar rk\bar k}+{\rm l.o.t.} \mbox{ at } p\,
$$
and the claim follows.
\end{proof}

Now we can prove the existence and uniqueness of a short-time solution to \eqref{flow}.
Since $Q$ is a quasilinear elliptic operator, theorem \ref{short time existence} implies that \eqref{scalarflow}
has a unique maximal solution $u$ and the corresponding $\omega_{u}$
solves \eqref{flow}. In order to prove that $\omega_u$ is unique, 
let us consider another solution $\omega_{\tilde u}$ to \eqref{flow}.  The function $\tilde u$ solves
\begin{equation}\label{eqninv}
\begin{cases}
\left(\partial_t  \tilde u_t\right)\omega^{n-2}=(s_{\tilde u_t})\, \omega^{n-2}+f_t\,\omega^{n-2}\\
\tilde u_{|t=0}=\tilde u_0\,,
\end{cases}
\end{equation}
where $f_t \in C^\infty(M)$ is smooth in $t$ and is such that
$$
\partial\bar\partial(f_t\,\omega^{n-2})=\partial\bar\partial(\tilde u_0\,\omega^{n-2})=0\,.
$$ 
Let 
$$
u'_t=\tilde u_t-\int_{0}^t f_s\,ds+(u_0-\tilde u_0)\,.
$$
Then $u'$ solves \eqref{scalarflow} and 
$$
\omega_{u'}=\omega_{\tilde u}\,.
$$
Since \eqref{scalarflow} has unique solution, we have $u'=u$ and consequently 
$$
\omega_{\tilde u}=\omega_{u}\,,
$$
as required. 

\section{Proof of theorem \ref{stability}}\label{section4}

In this section we prove theorem \ref{stability}. 
The argument of the proof is inspired by  \cite[section 8]{witt}, where it is proved the stability around torsion-free $G_2$-structures of a certain geometric flow. 

Accordingly to the set-up described in the introduction, let $W^{2r,2}_+(M)$ be an open neighbourhood of $0$ in $W^{2r,2}(M)$ which is invariant by translations by constants and 
let $Q\colon W^{2r,2}_+(M)\to L^{2}(M)$ be a smooth {\em elliptic} operator of order $2r$.
For $v\in  W^{2r,2}_+(M)$, we denote by 
$$
L_v=dQ_{v}\colon W^{2r,2}(M)\to L^{2}(M)
$$
the differential of $Q$ at $v$.  In order to simplify the notation, we write $L$ to denote the differential of $Q$ at $0$. 

\begin{lemma}\label{corIFT}
Let $l>\frac12{\rm dim}\,M+2r-1$. For every $T,\epsilon>0$ there exists  $\delta>0$,  depending on $\epsilon$,  $T$ and $l$, such that 
if $u_0\in W^{r(2l-1),2}(M)$ satisfies 
$$
\|u_0\|_{W^{r(2l-1),2}}<\delta,
$$
then \eqref{pproblem} has a solution $u$ defined in $M\times [0,T]$ such that 
\begin{equation}\label{primastima}
\|u\|_{P^l([0,T])}<\epsilon\,.
\end{equation}
\end{lemma}
\begin{proof}
Fix $T,\epsilon>0$. Using notation of theorem \ref{IFT}, we have 
$\mathcal F(0)=(0,0)$. The same theorem, via the implicit function theorem, implies that $\mathcal F$ is a homeomorphism from an open neighbourhood $\mathcal U$  of $0$ in $P^l([0,T])$ to a neighbourhood $\mathcal V$ of $(0,0)$ in $W^{r(2l-1),2}(M)\times P^{l-1}([0,T])$.
Now the continuity of $\mathcal F^{-1}$ means that for every $\epsilon>0$ we can find $\delta>0$ such that if $\|u_0\|_{W^{r(2l-1),2}(M)}<\delta$ we have $\|\mathcal F^{-1}(u_0,0)\|_{P^l([0,T])}<\epsilon$. Note that $\mathcal F^{-1}(u_0,0)$ is nothing but the solution $u$ of the problem \eqref{pproblem} with initial value $u_0$ and the claim follows.
\end{proof}

Note that the choice of $l$ in the previous lemma allows us to apply both the Sobolev embedding theorem and theorem \ref{IFT}. 
From now on we fix $k>\frac12{\rm dim}\,M+2r-1$ and consider $Q$ as an operator defined on $W^{k,2}_+(M)$.

Now we show that when the initial datum is close enough to $0$, then the solution to \eqref{parabolic} converges exponentially fast to  a point in $Q^{-1}(0)$. This part of the proof is obtained showing that $W^{k,2}$-norm of the solution decreases exponentially. 
 From now on we assume that $Q$ satisfies conditions (h1)--(h3) described in the introduction. 

\begin{lemma}
Near the origin the set $Q^{-1}(0)$ is a segment $(-a,a)$. 
\end{lemma}
\begin{proof}
Let $\tilde Q\colon W_0^{k,2}(M)\to W_0^{k-2r,2}(M)$ be defined as 
$$
\tilde Q(u)=Q(u)-\frac{1}{{\rm Vol}_g(M)}\int_M Q(u)\,dV_g\,.
$$
Then $\tilde Q$ is a differentiable operator whose derivative at $0$ is an isomorphism in view of  (h2).
%\tred{Non va bene, se le cose stanno cos\", la definizione la dobbiamo mettere nell'ipotesi (h2), non nella dimostrazione} 
The implicit function theorem implies that $\tilde Q$ is a bijection between an open  neighbourhood of $0$ in $W_0^{k,2}(M)$ and an open neighborhood of $0$ in $W_0^{k-2r,2}(M)$ and condition (h1) implies the statement. 
\end{proof}

Next we observe that $Q\colon W^{k,2}_+(M)\to L^{2}(M)$ is differentiable in the Fr\'echet sense.
\begin{lemma}
$Q\colon W^{k,2}_+(M)\to L^{2}(M)$ is Fr\'echet differentiable at $0$, i.e.  for every $\epsilon>0$ there exists $\delta>0$ such that if $v\in W^{k,2}_+(M)$ satisfies $\|v\|_{W^{k,2}}\leq \delta$, then
$$
\|Q(v)-L(v)\|_{L^{2}}\leq \epsilon \|v\|_{W^{k,2}}\,. 
$$   
\end{lemma}
\begin{proof}
In view of a classical result (see e.g. \cite{Vainberg}), it is enough to show that the G$\hat{{\rm a}}$teaux derivative of $Q$ is continuous at $0$. Namely, let $X=W^{k,2}(M)$ and $Y=L^2(M)$ and $\mathcal  L(X,Y)$ be the set of continuous linear maps from $X$ to $Y$.  Let $Q'\colon W_+^{k,2}(M)\to \mathcal  L(X,Y)$ be the map  $Q'(v)=L_v$ (note that $Q'(0)=L$). Then the continuity of $Q'$ at $0$ is equivalent to require that  for every $\epsilon>0$ there exists a positive $\delta$ such that 
$$
\|L_v(w)-L(w)\|_{L^2}\leq \epsilon\|w\|_{W^{k,2}} 
$$
for every  $w\in W^{k,2}(M)$ and $v\in W^{k,2}_+(M)$ such that $\|v\|_{W^{k,2}}\leq \delta $. The last inequality is implied by the smoothness of the coefficients of $Q$ 
and by the Sobolev embedding theorem which ensures that the coefficients of $L_v$ are continuous.
\end{proof}

The following lemma is based on a general theorem about symmetric $T$-bounded operators  (see theorem \ref{Tbound} in the appendix at the end of the present paper).

\begin{lemma}\label{4.4}
For every $\epsilon>0$ there exists $\delta>0$ such that if $v\in W^{k,2}_+(M)$ satisfies $\|v\|_{W^{k,2}}<\delta$, then 
\b\label{otto}
-\langle L_{v}(z),z\rangle_{L^2} \geq (1-\epsilon) \langle -L(z),z\rangle_{L^2}-\epsilon \|z\|^2_{L_2}
\e
for every $z\in W^{2r,2}(M)$. 
\end{lemma}
\begin{proof} For notation used in this proof see the appendix.
Fix $\epsilon >0$. Let $H=L^{2}(M)$ and consider the operators on $H$, 
$T:=-\epsilon L$ and $V_v:=L-\frac{1}{2}(L_v+L^*_v)$, with $v \in W^{k,2}_+(M)$. We take $D(T)=D(V_v)=W^{2r,2}(M)$. Condition (h3) implies that $T$ is symmetric and bounded from below (with $\gamma_T=0$). 
Elliptic regularity of $T$ implies also that there exists  $C>0$ such that $\|z\|_{W^{2r,2}} \leq C \|T(z)\|_{L^2}$ for all $z \in W^{2r,2}_0(M)$.
Moreover reasoning as in the proof of the previous lemma we deduce that $v \mapsto V_v$ is continuous as a map $W^{k,2}_+ \to \mathcal L(W^{2r,2},L^2)$. Now let us write $z=z_0+z_1$ according to the decomposition $W^{2r,2}(M)=\ker L \oplus (\ker L)^\perp$. Thus we can find $\delta>0$ such that if $\|v\|_{W^{k,2}}\leq \delta $, we have
$$
\|V_v(z_1)\|_{L^2}\leq b C^{-1} \|z_1\|_{W^{2r,2}} \leq b \|T(z)\|_{L^2}
$$
for every $z\in W^{k-2r,2}(M)$, with $b>0$ arbitrarily small. Consequently 
$$
\|V_v(z)\|_{L^2}\leq \frac12 \|L_v(z_0)+L^*_v(z_0)\|_{L^2} + \|V_v(z_1)\|_{L^2} \leq a\|z\|_{L^2} + b \|T(z)\|_{L^2}
$$
with $a>0$ arbitrarily small. Taking $a=\frac{\epsilon}{2}$, $b=\frac12$ and using \eqref{gamma} we have 
that $-\epsilon$ is a lower bound for $T+V$, hence the desired inequality.
\end{proof} 

Next we show that under our assumptions, $Q(u_t)$ has an $L^2$-exponential decay. 
From now on when $I$ is a time interval, we denote by $C^\infty_+(M\times I)$ the set $\{u \in C^\infty(M\times I)\colon u_t \in C_+^\infty(M) 
\,\, {\rm for} \,\, {\rm every} \,\, t\in I\}$.

\begin{lemma}\label{L2expdecay}
There exists $\delta>0$ such that if $u\in C^{\infty}_+(M\times [0,T])$ solves \eqref{parabolic} and satisfies
$$
\| u_t\|_{W^{k,2}}< \delta\,,\mbox{ for every }t\in [0,T]\,, 
$$
then 
$$
\|Q(u_t)\|_{L^2}^2\leq {\rm e}^{-\lambda_1 t} \|Q(u_0)\|_{L^2}^2\,,
$$
where $\lambda_1$ is the first positive eigenvalue of $-L$. 
\end{lemma}
\begin{proof}
We have
$$
\frac{d}{dt}\frac12 \|Q(u_t)\|_{L^2}^2=\langle L_{u_t}(Q(u_t)),Q(u_t) \rangle_{L^2}\,. 
$$
Hypothesis (h2) implies that  
$$
-\langle L(v),v\rangle_{L^2}\geq \lambda_1 \|v\|_{L^2}^2
$$
for all $v\in W^{k,2}_0(M)$. 
We can write $Q(u_t)=A_t+B_t$
according to the orthogonal splitting 
$L^2(M,\R)=\R\oplus L^2_0(M)$, i.e., 
$$
A_t=\frac{1}{{\rm Vol}_g(M)}\int_M Q(u_t)\,dV_g\,,\quad B_t=Q(u_t)-\frac{1}{{\rm Vol}_g(M)}\int_M Q(u_t)\,dV_g\,. 
$$
Then 
$$
-\langle L(Q(u_t)),Q(u_t)\rangle_{L^2}=-\langle L(B_t),B_t\rangle_{L^2}\geq \lambda_1
\|B_t\|^2_{L^2}
$$
which implies 
\b\label{-1r}
-\langle L(Q(u_t)),Q(u_t)\rangle_{L^2}\geq\lambda_1\left(\|Q(u_t)\|_{L^2}-\,{\rm Vol}_g(M)^{-\frac12}\,\left|\int_MQ(u_t)\,dV_g\right|\right)^2
\e 
Next we observe that for every $\epsilon >0$ there exists $\delta>0$ such that if $\|v\|_{W^{k,2}}<\delta $, then 
$$
\left|\int_M Q(v) dV_g \right|\leq \epsilon \, \|Q(v)\|_{L^2}\,.
$$
Indeed, since $Q(0)=0$ and $Q$ is Fr\'echet  differentiable at $0$, we have that for every $\epsilon>0$ 
there exists $\delta>0$ such that 
$$
\| Q(v)-L(v)\|_{L^{2}}\leq \epsilon  \|v\|_{W^{k,2}}
$$
for every $v\in W^{k,2}(M)$ such that $\|v\|_{W^{k,2}}<\delta$. 
Furthermore,  
elliptic regularity and assumption (h2) imply that 
$$
L\colon W^{k,2}_0(M)\to L^{2}_0(M)
$$
is an isomorphism on the image  and  then there exists a constant $C$ such that 
$$
\frac{1}{C}\|L(v)\|_{L^2}\leq  \|v\|_{W^{k,2}}\leq C\|L(v)\|_{L^{2}}
$$
fore every $v\in W^{k,2}_0(M)$. Therefore  
$$
\| Q(v)-L(v)\|_{L^{2}}\leq \epsilon  C\|L(v)\|_{L^{2}}
$$
for every $v\in W^{k,2}(M)$ such that $\|v\|_{W^{k,2}}<\delta$; and so, using (h2) again, we have 
\begin{multline*}
\left|\int_M Q(v) dV_g \right|= \left|\int_M [Q(v)-L(v)] dV_g \right|\leq \sqrt{{\rm Vol}_g(M)}\|Q(v)-L(v)\|_{L_2}\leq \\
\epsilon C \|L(v)\|_{L^2}\leq \frac{\epsilon C}{1-\epsilon C}\|Q(v)\|_{L^2}\,,
\end{multline*}
where to obtain the last inequality we used that 
$$
\|Q(v)\|_{L^2} \geq \|L(v)\|_{L^2}-\|Q(v)-L(v)\|_{L^2}\geq (1-\epsilon C)\|L(v)\|_{L^2}\,. 
$$

And so for a suitable choice of $\epsilon$ we find  $\delta>0$ such that if $\|v\|_{W^{k,2}}<\delta$, then 
$$
\left|\int_M Q(v) dV_g \right|\leq \frac{\sqrt{{\rm Vol}_g(M)}}{4}\|Q(v)\|_{L^2}\,,
$$   
and equation \eqref{-1r} implies 
$$
-\langle L(Q(u_t)),Q(u_t)\rangle_{L^2}\geq  \lambda_1 \frac{9}{16}\|Q(u_t)\|_{L^2}^2\,.
$$
Now lemma \ref{4.4} says that, for every $\epsilon' >0$, up to taking a smaller $\delta$ we have 
$$
\begin{aligned}
-\langle L_{u_t}Q(u_t),Q(u_t)\rangle_{L^2} \geq &\,(1-\epsilon')
\langle -L(Q(u_t)),Q(u_t)\rangle_{L^2}-\epsilon' \|Q(u_t)\|^2_{L_2}\\
\geq\,& 
\left((1-\epsilon')\lambda_1 \frac{9}{16}-\epsilon' \right)\|Q(u_t)\|_{L^2}^2\,.
\end{aligned}
$$
Choosing $\epsilon'$ small enough we obtain
$$
\langle L_{u_t}Q(u_t),Q(u_t)\rangle_{L^2} \leq  -\frac{\lambda_1}{2} \|Q(u_t)\|^2_{L^2}
$$
which yields
$$
\frac{d}{dt} \|Q(u_t)\|_{L^2}^2 \leq
-\lambda_1 \|Q(u_t)\|^2_{L^2}\,.
$$
Hence Gronwall's lemma implies
$$
\|Q(u_t)\|_{L^2}^2 \leq
{\rm e}^{-\lambda_1t} \|Q(u_0)\|^2_{L^2}\,,
$$
as required. 
\end{proof}

The following lemma is a version with compact time of lemma 7.13 in \cite{HP}.

%%%%%%%%%%%%%%%%%%%%%%%%%%%%%%%%%%%%%%%%%%%%%%%%%%%%%%%%%%

\begin{lemma} \label{indipHam}
Let $t_0\geq 0, T>0$ and $I=[t_0,t_0+T]$. 
For every $h>0$ there exist $l \geq k$ and $\delta>0$  such that 
for  $u\in C^{\infty}_+(M\times I) \cap B^{P^{l}(I)}(\delta)$ and $w\in P^{h+1}(I)$ we have 
\b
\label{viola}
\|w\|_{P^{h+1}(I)}\leq C\left(\|w_{t_0}\|_{W^{r(2h+1),2}}+\|\partial_t w-L_u(w)\|_{P^{h}(I)}\right)
\e
for some 
$C=C(n,h,T,\delta)$,
\end{lemma}

\begin{proof}
First of all we recall that by proposition 2.3 in \cite{MM} for every $h>0$, $z\in P^{h+1}(I)$ and $v\in P^{h}(I)$ we have 
\b\label{HPlemma7.13}
\| z\|_{P^{h+1}(I)}\leq C\left(\| z_{t_0}\|_{W^{r(2h+1),2}}+\|\partial_t  z-L_v( z)\|_{P^{h}(I)}\right)
\e
for some $C=C(n,h,T,v)$. \\
For every $l$ let us denote by $[L_u]_{P^{l}(I)}$ the norm of $L_u$ as an operator from $P^{l}(I)$ to $P^{l-1}(I)$. We first observe that if  $[L_u-L]_{P^{h}(I)}<\delta'$, for $\delta'$ small enough, then \eqref{HPlemma7.13} with $v=u$ holds with $C$ dependent on $\delta'$, but independent of $u$. This can be done as follows. Let $w\in P^{h+1}(I)$ and let 
$$
f=L_{u}(w)-\partial_tw\,. 
$$
Then $z=w$ is a solution to 
$$
\partial_tz=L(z)+(L_u-L)(w)+f
$$
and formula \eqref{HPlemma7.13} tells 
$$
\|w\|_{P^{h+1}(I)}\leq C\left(\|w_{t_0}\|_{W^{r(2h+1),2}}+\|(L_u-L)(w)+f\|_{P^{h}(I)}\right)
$$
with $C$ independent of $L_u$.
%\tred{Qua la cosa importante e' che la costante $Q$ nel lemma 7.13 di HP non dipenda da $g$, cioe' il termine noto dell'equazione altrimenti siam del gatto! (Cosi' sembrerebbe dal lemma 7.12) }.
Hence 
\begin{eqnarray*}
\|w\|_{P^{h+1}(I)} & \leq & C\left(\|w_{t_0}\|_{W^{r(2h+1),2}}+\|(L_u-L)(w)\|_{P^{h}(I)}+\|f\|_{P^{h}(I)}\right)\\
& \leq & C\left(\|w_{t_0}\|_{W^{r(2h+1),2}}+[L_u-L]_{P^{h}(I)}\, \|w\|_{P^{h+1}(I)}+\|f\|_{P^{h}(I)}\right)\\
& \leq & C\left(\|w_{t_0}\|_{W^{r(2h+1),2}}+\delta'\, \|w\|_{P^{h+1}(I)}+\|f\|_{P^{h}(I)}\right)\,,
\end{eqnarray*}
and for $\delta'$ small enough we get  
\begin{eqnarray*}
\|w\|_{P^{h+1}(I)} & \leq & \frac{C}{1-C\delta'}\left( \|w_{t_0}\|_{W^{r(2h+1),2}}+\|f\|_{P^{h}(I)} \right)\\
 & = & \frac{C}{1-C\delta'}\left(\|w_{t_0}\|_{W^{r(2h+1),2}}+\|\partial_tw-L_{u}(w)\|_{P^{h}(I)} \right)\,. 
\end{eqnarray*}

Since the coefficients of $L_u$ depend smoothly on $u$ and its space derivatives up to order $2r-1$, a suitable bound on $\|u\|_{C^{m}(M\times I)}$, for $m$ sufficiently large in terms of $h$, implies  that $[L_u-L]_{P^{h}(I)}<\delta'$. Using the parabolic Sobolev embedding theorem of \cite[Proposition 4.1]{MM} we can find $l \geq k$ depending on $m$ such that condition $u\in C^{\infty}_+(M\times I) \cap B^{P^{l}(I)}(\delta)$ with $\delta$ small enough implies $[L_u-L]_{P^{h}(I)}<\delta'$ and the claim follows. 
\end{proof}

%--------------------------------------------------------------------------------------------------------------
%------------------------------------------------------------------------------------------------------------------------

Now we can prove the following interior estimate for solutions of the linearized equation.

\begin{lemma}
\label{intest}
Let $I=[t_0,t_0+T] \subseteq \R_{\geq 0}$.
For all $\epsilon <T$ and $h \in \mathbb{N}$ there exist $l \geq k$ and  $\delta > 0$ such that for  $w\in P^{h}(I)$ and $u\in C^{\infty}_+(M\times I) \cap B^{P^{l}(I)}(\delta)$ and 
$$
\partial_tw=L_{u}w\,,
$$
we have
\b\label{induction}
\|w\|_{P^{h}([t_0+\epsilon,t_0+T])}\leq C\|w\|_{P^0(I)}\,. 
\e 
for some $C=C(n,h,\delta,T,\epsilon)$
\end{lemma}
\begin{proof}
We fix $0<\epsilon<T$ and we prove the statement by induction on $h$. For $h=0$ the claim is trivial. So we assume the statement true  up to $h=N$. 
Fix $w\in P^{N+1}(I)$ satisfying $\partial_tw=L_{u}w$.  Take a smooth cut-off function $\psi: [0,T] \to [0,1]$ such that  $\psi \equiv 0$ in  $[0,\epsilon/2]$ and $\psi \equiv 1$ in  $[\epsilon,T]$. 
Take $\chi(t)=\psi(t-t_0)$ and set $\tilde w_t(x)=\chi(t)w_t(x)$. Then $\partial_t\tilde w=\dot \chi w+\chi \partial_t w$ and 
$$
\partial_t \tilde w-L_u(\tilde w)=\dot \chi w
$$
and inequality \eqref{viola} implies  
$$
\|\tilde w\|_{P^{N+1}([t_0+\epsilon/2,t_0+T])}\leq C \|\dot \chi w\|_{P^{N}([t_0+\epsilon/2,t_0+T])} 
\leq C\,\| w\|_{P^{N}([t_0+\epsilon/2,t_0+T])}
$$
where in the second inequality we use a new constant $C$. 
Therefore 
$$
\|w\|_{P^{N+1}([t_0+\epsilon,t_0+T])}\leq\|\tilde w\|_{P^{N+1}([t_0+\epsilon/2,t_0+T])}\leq C\,\| w\|_{P^{N}([t_0+\epsilon/2,t_0+T])}\,.
$$
Now the claim follows by the induction assumption. 
\end{proof}

%----------------------------------------------
%----------------------------------------------

Now we are in a position to prove the $W^{k,2}$-exponential decay of $Q(u_t)$.

\begin{lemma}\label{Wexp}
Let $u\in C^{\infty}_+(M\times [0,T])$ be a solution to $\partial_t u_t=Q(u_t)$ in $[0,T]$ and let $\lambda_1$ be the first positive eigenvalue of $-L$.
Then for every $\epsilon \in (0,T/2)$ there exist $\delta >0$ and $l \geq k$ such that 
if $u \in B^{P^{l}([0,T])}(\delta)$, we have 
$$
\|Q(u_t)\|^2_{W^{k,2}}\leq C \|Q(u_0)\|^2_{L^2} \,{\rm e}^{-\lambda_1 t} \mbox{ for every }t\in [\epsilon,T-\epsilon]\,,
$$
for some positive $C=C(n,k,\delta,T,\epsilon)$.
\end{lemma}
\begin{proof}
First note that the assumption $u \in B^{P^{l}([0,T])}(\delta)$ implies, via corollary \ref{tracetheorem}, the pointwise bound $\|u_t\|_{W^{k,2}} \leq \delta'$ for every $t \in [0,T]$ for some $\delta'>0$ such that $\delta' \to 0$ when $\delta \to 0$.
In view of lemma \ref{L2expdecay} for a suitable choice of  $\delta>0$ we have that, 
$$
\|Q(u_t)\|_{L^2}^2\leq {\rm e}^{-\lambda_1 t} \|Q(u_0)\|_{L^2}^2\,  \quad \mbox{for every $t \in [0,T]$}.
$$
Integrating we get
$$
\int_{t}^{T} \|Q(u_\tau)\|_{L^2}^2\,d\tau\leq\frac{{\rm e}^{-\lambda_1t}}{\lambda_1}\|Q(u_0)\|_{L^2}^2\,.
$$
By differentiating $\partial_t u_t=Q(u_t)$ we get that $v=Q(u)$ solves 
$$
\partial_t v-L_{u}(v)=0\,.
$$
Hence we can apply lemma \ref{intest} and there exist $C=C(n,k,\delta,T,\epsilon)$ such that for $t \in [\epsilon/2,T]$
\begin{multline*}
\|Q(u)\|^2_{P^{\frac{k+r}{2r}}([t,T])}\leq C\|Q(u)\|^2_{P^0[t-\epsilon/2,T]}= \\
=C\int_{t-\epsilon/2}^{T}\|Q(u_\tau)\|_{L^2}^2\,d\tau \leq C\frac{{\rm e}^{-\lambda_1(t-\epsilon/2)}}{\lambda_1} \|Q(u_0)\|_{L^2}^2\,.
\end{multline*}
Furthermore, using corollary \ref{tracetheorem} with the same $\epsilon$, we have that
$$
\|Q(u_{t})\|_{_{W^{k,2}}}^2 \leq C \|Q(u)\|^2_{P^{\frac{k+r}{2r}}([t,T])}
$$
for every $t\in [\epsilon,T-\epsilon]$
and so 
$$
\|Q(u_{t})\|_{_{W^{k,2}}} ^2 \leq C \|Q(u_0)\|^2_{L^2}{\rm e}^{-\lambda_1 t}
$$
for every $t\in [\epsilon,T-\epsilon]$ which implies the statement. 
\end{proof}

\begin{proof}[Proof of theorem $\ref{stability}$] 
Let $T>0$ be fixed and $0<\epsilon<\frac{T}{2}$. In view of lemma \ref{corIFT} and lemma \ref{Wexp}, there exists $\delta>0$ such that  if $\|u_0\|_{W^{l,2}}\leq \delta$ for every $l \geq k$, then the solution $u_t$ to the geometric flow \eqref{parabolic} exists in $[0,T]$ and 
$$
\|Q(u_t)\|_{W^{k,2}}\leq C\|Q(u_0)\|_{L^2}{\rm e}^{-\lambda t/2} \mbox{ for every }t\in [\epsilon,T-\epsilon]\,,
$$
for some $C=C(n,k,\delta,T,\epsilon)>0$.  Choose $u_0$ and $\epsilon$ such that  $\|u_0\|_{W^{k,2}}\leq \delta$ and 
\begin{equation}\label{condizione}
C\|Q(u_0)\|_{L^2}\frac{{\rm e}^{-\lambda \epsilon}}{\lambda}\sum_{h=0}^{\infty}{\rm e}^{-\lambda h(T-2\epsilon)}+\|u_{\epsilon}\|_{W^{k,2}}\leq \delta\,.  
\end{equation}
Now we show that $u$ can be extended in $M\times [0,\infty)$. 
%and converges to ????  for to $t\to \infty$.  
For $t \in [\epsilon,T-\epsilon]$ we have 
$$
\begin{aligned}
\|u_{t}\|_{W^{k,2}}=&\left\|\int_{\epsilon}^{t}Q(u_\tau)\,d\tau+u_\epsilon\right\|_{W^{k,2}}\leq \int_{\epsilon}^{t}\|Q(u_\tau)\|_{W^{k,2}}\,d\tau+\|u_\epsilon\|_{W^{k,2}}\\
&\leq C\|Q(u_0)\|_{L^2}\frac{{\rm e}^{-\lambda_1\epsilon}}{\lambda_1}+\|u_\epsilon\|_{W^{k,2}}\,
\end{aligned}
$$
and condition \eqref{condizione} implies $\|u_t\|_{W^{k,2}}\leq \delta$. \\
\indent In particular, $\|u_{T-\epsilon}\|_{W^{k,2}}\leq \delta $ and $u$ can be extended  in $M\times [0,2T-2\epsilon]$. Since also $\|u_{T-2\epsilon}\|_{W^{k,2}}\leq \delta $ we have 
$$
\|Q(u_t)\|_{W^{k,2}}\leq C \|Q(u_0)\|_{L^2} \,{\rm e}^{-\lambda_1 t} \,,\mbox{ for every }t\in [T-\epsilon,2T-3\epsilon]
$$ 
and so for $t\in [T-\epsilon, 2T-3\epsilon]$ we have 
$$
\begin{aligned}
\|u_t\|_{W^{k,2}}=&\left\|\int_{T-\epsilon}^{t}Q(u_\tau)\,d\tau+u_{T}\right\|_{W^{k,2}}\leq \int_{T-\epsilon}^{t}\|Q(u_\tau)\|_{W^{k,2}}\,d\tau+\|u_{T-\epsilon}\|_{W^{k,2}}\\
&\leq C\|Q(u_0)\|_{L^2}\left(\frac{{\rm e}^{-\lambda_1 (T-\epsilon)}}{\lambda_1}+\frac{{\rm e}^{-\lambda_1 \epsilon}}{\lambda_1}\right)+\|u_{\epsilon}\|_{W^{k,2}}\leq \delta\,.
\end{aligned}
$$
In particular, $\|u_{2T-3\epsilon}\|_{W^{k,2}}<\delta$ and $u$ can be extended in $M\times [0,3T-3\epsilon]$.  Moreover since $\|u_{2T-4\epsilon}\|<\delta$ for $t\in [2T-3\epsilon,3T-5\epsilon]$ we have   
$$
\begin{aligned}
\|u_t\|_{W^{k,2}}=&\left\|\int_{2T-3\epsilon}^{t}Q(u_\tau)\,d\tau+u_{2T-3\epsilon}\right\|_{W^{k,2}}\leq \int_{2T-3\epsilon}^{t}\|Q(u_\tau)\|_{W^{k,2}}\,d\tau+\|u_{2T-3\epsilon}\|_{W^{k,2}}\\
&\leq C\|Q(u_0)\|_{L^2} \left(\frac{{\rm e}^{-\lambda_1 (2T-3\epsilon)}}{\lambda_1}+\frac{{\rm e}^{-\lambda_1 (T-\epsilon)}}{\lambda_1}+\frac{{\rm e}^{-\lambda_1 \epsilon}}{\lambda}\right) +\|u_\epsilon\|_{W^{k,2}}\leq \delta\,. 
\end{aligned}
$$
By iterating this procedure yields that for every positive integer $N$ we have the estimate
$$
\|u_{(N+1)T-(2N+1)\epsilon}\|_{W^{k,2}}\leq C\|Q(u_0)\|_{L^2} \frac{{\rm e}^{-\lambda_1 \epsilon}}{\lambda_1} \sum_{h=0}^{N}{\rm e}^{-\lambda h(T-2\epsilon)} +\|u_\epsilon\|_{W^{k,2}}\leq \delta\,.
$$
Therefore the maximal solution to \eqref{parabolic} is defined in $M\times [0,\infty)$. 
Finally we define 
$$
u_{\infty}:=u_0+\int_{0}^{\infty} Q(u_t)\,dt\,.
$$
Since 
$$
u_t-u_\infty=\int_0^tQ(u_\tau)\,d\tau-\int_{0}^{\infty} Q(u_\tau)\,d\tau=-\int_t^{\infty} Q(u_\tau)\,d\tau\,,
$$
we have 
$$
\|u_t-u_\infty\|_{W^{k,2}}= \left\|\int_{t}^\infty Q(u_\tau)\,d\tau \right\|_{W^{k,2}}\leq \int_{t}^\infty \|Q(u_\tau)\|_{W^{k,2}}\,d\tau \leq C\|Q(u_0)\|_{L^2}^2{\rm e}^{-\lambda_1t/2}\,,
$$
and $u_t$ converges exponentially fast to $u_{\infty}$ in $W^{k,2}$-norm. Finally, since $Q$ is continuous, we have 
$$
\|Q(u_{\infty})\|_{W^{k,2}}=\|\lim_{t\to \infty}Q(u_t)\|_{W^{k,2}}\leq \lim_{t\to \infty}\|Q(u_t)\|_{W^{k,2}}=0
$$
and then $Q(u_\infty)=0$, as required.
 
Finally to obtain the claim it is enough to observe that $k>\frac12{\rm dim}\,M+2r-1$ is arbitrary.
\end{proof}

\section{Stability of Calabi-type flows}
\subsection{Stability of the classical Calabi flow around  cscK metrics}
In \cite{Chen-He} Chen and He proved the stability of the Calabi flow around cscK metrics. 
Here we observe that theorem \ref{stability} can be used to obtain an alternative proof of the 
same result.

Let $(M,g)$ be a compact K\"ahler manifold with fundamental form $\omega$.  The {\em Calabi flow} is the geometric flow of K\"ahler forms governed by the following equation  
\b\label{CF}
\partial_t\omega_t=\sqrt{-1}\partial\bar\partial R_t\,,\quad \omega_{|t=0}=\omega_0
\e   
where $R_t$ is the scalar curvature of $\omega_t$ and the initial $\omega_0$ lies in the same cohomology class as $\omega$. 
As an application of our Theorem \ref{main} we can give an alternative proof of Theorem 4.1 in \cite{Chen-He}.
We restate it for the sake of completeness
\begin{theorem}[Chen-He]
\label{classtab}
Let $(M,\omega)$ be a  compact K\"ahler manifold with constant scalar curvature. Then there exists $\delta>0$ such that if $\omega_0$ is a K\"ahler metric satisfying 
$$
\|\omega_0-\omega\|_{C^{\infty}}<\delta\,,
$$  
then the Calabi-flow starting from $\omega_0$ is immortal and converges in $C^{\infty}$ topology to a constant scalar curvature K\"ahler metric in $[\omega].$ 
 \end{theorem}
\begin{proof}
Let 
$$
C^{\infty}_+(M)=\{u\in C^{\infty}(M)\,\,:\,\, \omega+\sqrt{-1}\partial \bar\partial u>0\}\,.
$$
We can rewrite the Calabi-flow \eqref{CF} in terms of K\"ahler potential as 
$$
\partial_tu_t=R_{t}-R\,,\quad u_{|t=0}=u_0
$$
where $\omega_0=\omega+i\partial \bar\partial u_0$, $R_t$ is  the scalar curvature of $\omega_t=\omega+i\partial \bar \partial u_t$
and $R$ is the scalar curvature of $\omega$.
Consider the operator
$$
Q(v)=R_v-R
$$
defined on K\"ahler potentials, where $R_v$ is the scalar curvature of $\omega+i\partial \bar \partial v$. 
We observe that $Q$ satisfies conditions (h1)--(h3) described in the introduction. 
It is obvious that $Q$ satisfies (h1). On the other hand since $R$ is constant we have 
$$
L(v)=Q_{|*0}(v)=-\mathcal D^*\mathcal Dv
$$ 
for every smooth function $v$, where $\mathcal D^*\mathcal D$  is the Lichnerowicz operator induced by $\omega$
(see e.g. \cite[Section 4]{Szeke}). It follows that $Q$ satisfies also conditions (h2) and (h3) and therefore theorem \ref{stability} implies the statement. 
\end{proof}
We observe that proposition \ref{classtab} is also implied by theorems 1 and 2 in \cite{TW0}. 
Moreover it is worth noting that the stability of the Calabi flow has been extended to extremal metrics in \cite{HZ}.

%%%%%%%%%%%%%%%%%%%%%%%%%%%%%%%%%%%%%%%%%%%%%%%%%%%%%%%%%%%%%%%%%

\subsection{Stability of \eqref{flow} near cscK metrics}\label{stabilitysection}
%$k\geq 8$ [per la dimostrazione del corollario 6.2]
Now we prove the second part of theorem \ref{main} about the stability of \eqref{flow} near constant scalar curvature K\"ahler (cscK) metrics. 

On a compact complex manifold $M$ with cscK metric $\omega$  consider the parabolic flow \eqref{flow} with background form $\omega$. In this case the flow reduces to 
\b\label{flowk}
\frac{\partial}{\partial t}\omega_t^{n-1}=\sqrt{-1}\partial \bar \partial(s_{\omega_t})\wedge\omega^{n-2}\,,\quad 
 \omega_{|t=0}=\omega_0
\e
since $\omega$ is closed. Here it is convenient to reduce the balanced flow to a scalar one by using the  substitution 
$\omega_t=\omega^{n-1}+i\partial \bar\partial (u_t-R)\wedge \omega^{n-1}$ where $R$ is the scalar curvature of $\omega$. This leads to consider the parabolic equation 
$$
\partial_tu_t=s_{u_t}-R\,,\quad u_{t=0}=u_0\,. 
$$
Let $Q\colon C^\infty_\omega(M)\to C^\infty(M)$ be the operator 
$Q(v)=s_v-R$. As usual denote by $L$ the differential of $Q$ at $ 0$. Note that since $\omega$ is closed, $Q$ obviously satisfies hypothesis (h1). Next we show that $Q$ satisfies also conditions (h2) and (h3) described in the introduction and then we apply theorem \ref{parabolic}. 
Let $f \in C^{\infty}_\omega(M\times(-\delta,\delta))$ be an arbitrary time-dependent smooth function such that $f_0=0$. Let $v=\partial_t f_{|t=0}$.   We denote by $\dot \omega$ the time derivative of $\omega_{f_t}$ at $t=0$.  Since $\omega$ is closed we have 
$$
(n-1)\dot \omega\wedge \omega^{n-2}=\partial_{t|t=0}{\omega^{n-1}_f}= \sqrt{-1}\partial \bar \partial v\wedge \omega^{n-2}\,,
$$
which implies 
$$
\dot \omega=\frac{\sqrt{-1}}{n-1}\partial \bar \partial v\,.
$$
Now 
$$
Q(f_t)=s_{f_t}-R=-g_{f_t}^{\bar b a}\partial_a \bar\partial_b \log\det g_{f_t}-R
$$
and in analogy with the K\"ahler case 
%taking into account that $\omega$ has constant scalar curvature 
we have 
\begin{eqnarray*}
L(v) & = & 
\partial_{t|t=0}Q(f_t)=
-\frac{1}{n-1} g^{\bar b k}\partial_k\bar\partial_l v g^{\bar l a} \rho_{a\bar b}- \frac{1}{n-1}g^{\bar b a}\partial_a \bar\partial_bg^{\bar m l}\partial_l \partial_{\bar m}v \\
%=-\frac{1}{n-1}\Delta\Delta v+\frac{R}{n-1} \Delta v \\
 & = & -\frac{1}{n-1} \mathcal{D}^*\mathcal{D}v
\end{eqnarray*}
where $\rho$ is the Ricci form of $\omega$ and $\mathcal{D}^*\mathcal{D}$ is the Lichnerowicz operator, see again \cite{Szeke}.
Thus $Q$ satisfies all the hypotheses of theorem \ref{stability} and the result follows.

%%%%%%%%%%%%%%%%%%%%%%%%%%%%%%%%%%%%%%%%%%%%%%%%%%%%%%%%%%%%%

\section{remarks}
It is rather natural to compare the flow introduced in this paper with other flows of Hermitian metrics considered in literature.  
For instance in \cite{BV} the authors introduced the {\em balanced flow} 
$$
\partial_t\omega_t^{n-1}=\sqrt{-1}(n-1)!\partial \bar\partial *_t({\rm Ric}^C_{\omega_t}\wedge \omega_t)+(n-1)\Delta_{BC}\omega_{t}^{n-1}\,,\quad \omega_{|t=0}=\omega_0
$$
which evolves an initial balanced form in its Bott-Chern class and in \cite{TW} Tosatti and Weinkove introduced the 
$(n-1)$-{\em plurisubharmonic flow}
\b\label{pluri}
\partial_t\omega_t^{n-1}=-(n-1){\rm Ric}^C_{\omega_t}\wedge \omega^{n-2}\,,\quad \omega_{|t=0}=\omega_0
\e
which depends on the choice of a background form $\omega$ (the flow was subsequently studied by Gill in \cite{gill}). In both flows ${\rm Ric}^C_{\omega_t}$ is the Chern Ricci form of $\omega_t$, while in the the first one $\Delta_{BC}$ is the Bott-Chern Laplacian of $\omega_t$.   
Flow \eqref{flow} and the balanced flow can be both seen as a generalisation of the classical Calabi flow (see e.g. \cite{Chen-He}) to the non-K\"ahler case, but they are in fact different in many aspects. Firstly, \eqref{flow} is always a potential flow, while in general the balanced flow cannot be reduced to a scalar equation; \eqref{flow} does not preserve the K\"ahler condition, whilst when the initial form is K\"ahler the balanced flow reduces to the classical Calabi flow; 
%in contrast to the case of the balanced flow, 
the definition of \eqref{flow} depends on a fixed background metric; moreover as far as we know, the balanced flow equation is well-posed only when the initial form is balanced, while theorem \ref{main} in the present paper says that \eqref{flow} is always well-posed; finally, when the initial $\omega_0$  is balanced, then \eqref{flow} evolves $\omega_0$ in $\mathcal C_{\omega}(M)$, while in general the balanced flow moves $\omega_0$ in its Bott-Chern class, but outside  $\mathcal C_{\omega}(M)$. 

The $(n-1)$-{\em plurisubharmonic flow} is in some aspects similar to the flow considered in this paper, since its definition depends on the choice of a background metric. But in contrast to \eqref{flow}, \eqref{pluri} does not necessarily preserve the balanced condition if the background metric is non-K\"ahler.
  
As a final remark we note that recently Zheng introduced in \cite{tao} a new second order flow of Gauduchon metrics whose definition depends on the choice of a background Gauduchon metric.

\section*{Appendix: Symmetric $T$-bounded operators}
%Let us recall some basic facts about linear operators on Hilbert spaces. 
We refer to \cite{weidmann} for a detailed description of the following topics. Let $(H, \langle\cdot,\cdot\rangle)$ be a Hilbert space.  A {\em linear operator} $T$ on $H$ is a linear map from a linear subspace $D(H)$ of $H$ (called the {\em domain of $T$}) into $H$. $T$ is {\em symmetric} if $D(T)$ is dense in $H$ and 
$$
\langle h_1,T(h_2)\rangle=\langle T(h_1),h_2\rangle\quad \mbox{ for every } h_1,h_2\in D(T)\,. 
$$
Moreover $T$ is called {\em self-adjoint} if it is symmetric and  
$$
D^*(T)=\{h\in H\,\,:\,\, \mbox{ the map } f\mapsto \langle h,T(f)\rangle\,\, \mbox{ is continuous on } D(T)\}
$$ 
coincides with $D(T)$. 
The self-adjoint condition is in general more restrictive than the symmetry since $D^*(T)$ can properly  contain $D(T)$; for instance when $T$ is bounded, then $D^*(T)$ is always the whole $H$. Furthermore $T$ is said to be {\em bounded from below} if there exists a constant $\gamma$ such that
$$
\langle h,T(h)\rangle \geq \gamma \|h\|, \quad \mbox{ for every } h\in D(T)\,.
$$
Let us consider now two linear operators $T,V$  in $H$. $V$ is called {\em $T$-bounded} if $D(T)\subseteq D(V)$ and there exists a constant $C$ such that
$$
\|V(h)\|\leq C \,\sqrt{\|h\|^2+\|T(h)\|^2}, \quad \mbox{ for every } h\in D(T)\,.
$$
If $V$ is $T$-bounded, then 
$$
\|V(h)\|\leq C(\|h\|+\|T(h)\|), \quad \mbox{ for every } h\in D(T)
$$
and the infimum of all numbers $b\geq 0$ for which there exists an $a\geq 0$ such that 
$$
\|V(h)\|\leq a\|h\|+b\|T(h)\|, \quad\mbox{ for every } h\in D(T)
$$
is called the {\em $T$-bound} of $V$. Note that if $V$ is bounded (i.e. there exists a constant $C$ such that 
$\|V(h)\|\leq C\|h\|$), then $V$ is also $T$-bounded  with bound $0$ by every operator $T$ on $H$. 
In the paper we use the following (see e.g. \cite[Theorem 9.1]{weidmann})
\begin{theorem}\label{Tbound}
Let   $T,V$ be {\em linear operators} on $H$. Assume $T$ self-adjoint and bounded from below with lower bound $\gamma_T$ and  $V$ symmetric and $T$-bounded with $T$-bound $<1$. Then $T+V$ is self-adjoint and bounded from below.  Moreover if 
$$
\|V(h)\|\leq a\|h\|+b\|T(h)\|, \quad\mbox{ for every } h\in D(T),
$$
where $b<1$, then 
\b\label{gamma}
\gamma= \gamma_T-{\rm max}\left\{\frac{a}{1-b}, a+b|\gamma_T| \right\}
\e
is a lower bound of $T+V$. 
\end{theorem}

\end{document}